\documentclass[reqno]{amsart}
\usepackage{amsmath,amssymb}
\usepackage{graphicx}
\usepackage{pictexwd,dcpic,amssymb}

\newtheorem{theorem}{Theorem}[section]
\newtheorem{proposition}[theorem]{Proposition}

\theoremstyle{definition}

\theoremstyle{remark}

\numberwithin{equation}{section}

\begin{document}

\title{A note on traces of singular moduli}

\author{Ja Kyung Koo}
\address{Department of Mathematical Sciences, KAIST}
\curraddr{Daejeon 373-1, Korea} \email{jkkoo@math.kaist.ac.kr}
\thanks{}

\author{Dong Hwa Shin}
\address{Department of Mathematical Sciences, KAIST}
\curraddr{Daejeon 373-1, Korea} \email{shakur01@kaist.ac.kr}
\thanks{}


\keywords{Singular moduli, modular forms, congruences.
\newline The first named author is partially supported by Basic Science Research Program through the NRF of Korea
funded by the MEST (2010-0001654).}

\maketitle

\begin{abstract}
We will generalize Osburn's work (\cite{Osburn}) about a congruence
for traces defined in terms of Hauptmodul associated to certain
genus zero groups of higher levels.
\end{abstract}

\section {Introduction}

Let $\mathfrak{H}$ denote the complex upper half-plane and
$\mathfrak{H}^*:=\mathfrak{H}\cup\mathbb{Q}\cup\{\infty\}$. For an
integer $N\geq2$ let $\Gamma_0(N)^*$ be the group generated by
$\Gamma_0(N)$ and all Atkin-Lehner involutions $W_e$ for $e||N$.
There are only finitely many $N$ for which the modular curve
$\Gamma_0(N)^*\backslash\mathfrak{H}^*$ has genus $0$ (\cite{Ogg}).
In particular, if we let $\mathfrak{S}$ be the set of such $N$ which
are prime, then
\begin{equation*}
\mathfrak{S}=\{2,3,5,7,11,13,17,19,23,29,31,41,47,59,71\}.
\end{equation*}
For each $p\in\mathfrak{S}$ let $j_p^*(\tau)$ be the corresponding
Hauptmodul with a Fourier expansion of the form $q^{-1}+O(q)$ where
$q:=e^{2\pi i\tau}$.
\par
Let $p\in\mathfrak{S}$. For an integer $d\geq1$ such that
$-d\equiv\Box\pmod{4p}$ let $\mathcal{Q}_d$ be the set of all
positive definite integral binary quadratic forms
$Q(x,y)=[a,b,c]=ax^2+bxy+cy^2$ of discriminant $-d=b^2-4ac$. To each
$Q\in\mathcal{Q}_d$ we associate the unique root
$\alpha_Q\in\mathfrak{H}$ of $Q(x,1)$. Consider the set
\begin{equation*}
\mathcal{Q}_{d,p}:=\{[a,b,c]\in\mathcal{Q}_d~:~a\equiv0\pmod{p}\}
\end{equation*}
on which $\Gamma_0(p)^*$ acts. We then define the \textit{trace}
$t^{(p)}(d)$ by
\begin{equation*}
t^{(p)}(d):=\sum_{Q\in\mathcal{Q}_{d,p}/\Gamma_0(p)^*}\frac{1}{\omega_Q}
j_p^*(\alpha_Q)\in\mathbb{Z}
\end{equation*}
where $\omega_Q$ is the number of stabilizers of $Q$ in the
transformation group $\pm\Gamma_0(p)^*/\pm1$ (\cite{Kim}).
\par
Osburn (\cite{Osburn}) showed the following congruence:

\begin{theorem}\label{Osburntheorem}
Let $p\in\mathfrak{S}$. If $d\geq1$ is an integer such that
$-d\equiv\Box\pmod{4p}$ and $\ell\neq p$ is an odd prime which
splits in $\mathbb{Q}(\sqrt{-d})$, then
\begin{equation*}
t^{(p)}(\ell^2 d)\equiv0\pmod{\ell}.
\end{equation*}
\end{theorem}

Although this result is true, we think that his proof seems to be
unclear. Precisely speaking, let $D\geq1$ be an integer such that
$D\equiv\Box\pmod{4p}$. In $\S$3 we shall define
\begin{eqnarray*}
A_{\ell}(D,d)&:=&\textrm{the coefficient of $q^D$ in
$f_{d,p}(\tau)|T_{1/2,p}(\ell^2)$}\\
B_{\ell}(D,d)&:=&\textrm{the coefficient of $q^d$ in
$g_{D,p}(\tau)|T_{3/2,p}(\ell^2)$}
\end{eqnarray*}
where $f_{d,p}(\tau)$ and $g_{D,p}(\tau)$ are certain half integral
weight modular forms, and $T_{1/2,p}(\ell^2)$ and
$T_{3/2,p}(\ell^2)$ are Hecke operators of weight $1/2$ and $3/2$,
respectively. The key step that is not presented in Osburn's work is
the fact $A_\ell(1,d)=-B_\ell(1,d)$ which would be nontrivial at
all. In this paper we shall first give a proof of more general
statement $A_\ell(D,d)=-B_\ell(D,d)$ (Proposition \ref{main}), and
then further generalize Theorem \ref{Osburntheorem} as follows,
\begin{equation*}
t^{(p)}(\ell^{2n}d)\equiv0\pmod{\ell^n}
\end{equation*}
for all $n\geq1$ (Theorem \ref{generalize}).

\section {Preliminaries}

Let $k$ and $N\geq1$ be integers. If $f(\tau)$ is a function on
$\mathfrak{H}$ and $\gamma=\begin{pmatrix}
a&b\\c&d\end{pmatrix}\in\Gamma_0(4N)$, then we define the
\textit{slash operator} $[\gamma]_{k+1/2}$ on $f(\tau)$ by
\begin{equation*}
f(\tau)|[\gamma]_{k+1/2}:=j(\gamma,\tau)^{-2k-1}f(\gamma\tau)
\end{equation*}
where
\begin{equation*}
j(\gamma,\tau):=\bigg(\frac{c}{d}\bigg)\varepsilon_d^{-1}\sqrt{c\tau+d}\quad
\textrm{with}~\varepsilon_d:=\left\{\begin{array}{ll} 1
&\textrm{if}~d\equiv1\pmod4\\
i&\textrm{if}~d\equiv3\pmod4.\end{array}\right.
\end{equation*}
Here $(\frac{c}{d})$ is the Kronecker symbol and $\sqrt{c\tau+d}$
takes its argument on the interval $(-\pi/2,\pi/2]$.
\par
We denote by $M_{k+1/2}^{+\cdots+}(N)^!$ the infinite dimensional
vector space of weakly holomorphic modular forms of weight $k+1/2$
on $\Gamma_0(4N)$ which satisfy the Kohnen plus condition. Namely,
the space consists of the functions $f(\tau)$ on $\mathfrak{H}$ such
that
\begin{itemize}
\item[(i)] $f(\tau)$ is holomorphic on $\mathfrak{H}$ and
meromorphic at the cusps;
\item[(ii)] $f(\tau)$ is invariant under the action of
$[\gamma]_{k+1/2}$ for all $\gamma\in\Gamma_0(4N)$;
\item[(iii)] $f(\tau)$ has a Fourier expansion of the form
\begin{equation*}
\sum_{(-1)^kn\equiv\Box\pmod{4N}}a(n)q^n.
\end{equation*}
\end{itemize}
\par
Suppose that $\ell$ is a prime with $\ell\nmid N$. The action of the
\textit{Hecke operator} $T_{k+1/2,N}(\ell^2)$ on a form
\begin{equation*}
f(\tau)=\sum_{(-1)^k
n\equiv\Box\pmod{4N}}a(n)q^n\quad\textrm{in}\quad
M_{k+1/2}^{+\cdots+}(N)^!
\end{equation*}
is given by
\begin{equation}\label{Heckeoperator}
f(\tau)|T_{k+1/2,N}(\ell^2):=
\ell_k\sum_{(-1)^kn\equiv\Box\pmod{4N}}
\bigg(a(\ell^2n)+\bigg(\frac{(-1)^k n}{\ell}\bigg)\ell^{k-1}a(n)
+\ell^{2k-1}a(n/\ell^2)\bigg)q^n
\end{equation}
where
\begin{equation*}
\ell_k:=\left\{\begin{array}{ll} \ell^{1-2k}&\textrm{if $k\leq0$}\\
1&\textrm{otherwise}.
\end{array}\right.
\end{equation*}
Here $a(n/\ell^2):=0$ if $\ell^2\nmid n$. As is well-known,
$f(\tau)|T_{k+1/2,N}(\ell^2)$ belongs to
$M_{k+1/2}^{+\cdots+}(N)^!$.

\begin{proposition}\label{Jacobi}
Let $p\in\mathfrak{S}$.
\begin{itemize}
\item[(i)] For every integer $D\geq1$ such that
$D\equiv\Box\pmod{4p}$ there is a unique $g_{D,p}$ in
$M_{3/2}^{+\cdots+}(p)^!$ with the Fourier expansion
\begin{equation*}
g_{D,p}(\tau)=q^{-D}+\sum_{d\geq0,~-d\equiv\Box\pmod{4p}}B(D,d)q^d
\qquad(B(D,d)\in\mathbb{Z}).
\end{equation*}
\item[(ii)] For every integer $d\geq0$ such that $-d\equiv\Box\pmod{4p}$ there is
a unique form
\begin{equation*}
f_{d,p}(\tau)=\sum_{D\in\mathbb{Z}}A(D,d)q^D\qquad(A(D,d)\in\mathbb{Z})
\end{equation*}
in $M_{1/2}^{+\cdots+}(p)^!$ with a Fourier expansion of the form
$q^{-d}+O(q)$. They form a basis of $M_{1/2}^{+\cdots+}(p)^!$.
\item[(iii)] For every integer $d\geq0$ such that $-d\equiv\Box\pmod{4p}$ and
every integer $D\geq1$ such that $D\equiv\Box\pmod{4p}$ we have
\begin{equation*}
A(D,d)=-B(D,d).
\end{equation*}
\item[(iv)] For every integer $d\geq1$ such that
$-d\equiv\Box\pmod{4p}$ we get
\begin{equation*}
t^{(p)}(d)=-B(1,d).
\end{equation*}
\end{itemize}
\end{proposition}
\begin{proof}
See \cite{EZ} Theorem 5.6, \cite{Kim0} $\S$2.2, \cite{Kim} Lemma 3.4
and Corollary 3.5.
\end{proof}

\section {Generalization of Theorem \ref{Osburntheorem}}

We first prove the following necessary proposition by adopting
Zagier's argument (\cite{Zagier} Theorem 5).

\begin{proposition}\label{main}
Let $p\in\mathfrak{S}$ and $\ell\neq p$ be a prime. For each integer
$d\geq0$ such that $-d\equiv\Box\pmod{4p}$, define integers
$A_\ell(D,d)$ and $B_\ell(D,d)$ in the following manner:
\begin{eqnarray*}
A_{\ell}(D,d)&:=&\textrm{the coefficient of $q^D$ in
$f_{d,p}(\tau)|T_{1/2,p}(\ell^2)$}~\textrm{for each integer $D$}\\
B_{\ell}(D,d)&:=&\textrm{the coefficient of $q^d$ in
$g_{D,p}(\tau)|T_{3/2,p}(\ell^2)$}~\textrm{for each integer $D\geq1$}\\
&&\textrm{such that $D\equiv\Box\pmod{4p}$}.
\end{eqnarray*}
Then we have the relation
\begin{equation*}
A_{\ell}(D,d)=-B_{\ell}(D,d)~\textrm{for every integer $D\geq1$ such
that $D\equiv\Box\pmod{4p}$}.
\end{equation*}
\end{proposition}
\begin{proof}
For a pair of rational numbers $a$ and $b$, let
\begin{equation*}
\delta_{a,b}:=\left\{\begin{array}{ll}1 & \textrm{if}~a=b\in\mathbb{Z}\\
0 & \textrm{otherwise}.\end{array}\right.
\end{equation*}
Let $d\geq0$ be a fixed integer such that $-d\equiv\Box\pmod{4p}$.
It follows from the defining property of $f_{d,p}(\tau)$, namely
\begin{equation*}
A(D,d)=\delta_{D,-d}\quad\textrm{if}~D\leq0
\end{equation*}
that if $D\leq0$, then
\begin{eqnarray*}
A_\ell(D,d)&=&\ell
A(\ell^2D,d)+\bigg(\frac{D}{\ell}\bigg)A(D,d)+A(D/\ell^2,d)~\textrm{by the definition (\ref{Heckeoperator})}\\
&=&\ell
\delta_{\ell^2D,-d}+\bigg(\frac{D}{\ell}\bigg)\delta_{D,-d}+\delta_{D/\ell^2,-d}\\
&=&\ell\delta_{D,-d/\ell^2}+\bigg(\frac{D}{\ell}\bigg)\delta_{D,-d}+\delta_{D,-d\ell^2}.
\end{eqnarray*}
Hence the principal part of $f_{d,p}(\tau)|T_{1/2,p}(\ell^2)$ at
infinity is
\begin{equation*}
\ell q^{-d/\ell^2}+\bigg(\frac{-d}{\ell}\bigg)q^{-d}+q^{-d\ell^2}
\end{equation*}
where the first term should be omitted unless $-d/\ell^2$ is an
integer. Therefore we achieve
\begin{equation}\label{combination}
f_{d,p}(\tau)|T_{1/2,p}(\ell^2)= \ell
f_{d/\ell^2,p}(\tau)+\bigg(\frac{-d}{\ell}\bigg)f_{d,p}(\tau)+f_{d\ell^2,p}(\tau)~\textrm{by
Proposition \ref{Jacobi}(ii)}.
\end{equation}
And, for every integer $D\geq1$ such that $D\equiv\Box\pmod{4p}$ we
derive that
\begin{eqnarray*}
A_\ell(D,d)&=&\ell
A(D,d/\ell^2)+\bigg(\frac{-d}{\ell}\bigg)A(D,d)+A(D,d\ell^2)~\textrm{by (\ref{combination})}\\
&=&-\ell
B(D,d/\ell^2)-\bigg(\frac{-d}{\ell}\bigg)B(D,d)-B(D,d\ell^2)
~\textrm{by Proposition \ref{Jacobi}(iii)}\\
&=&-B_\ell(D,d)~\textrm{by the definition (\ref{Heckeoperator})}.
\end{eqnarray*}
\end{proof}

On the other hand, we apply Jenkins' idea (\cite{Jenkins}) to
develop a formula for the coefficient $B(D, \ell^{2n}d)$.

\begin{proposition}\label{Jenkins}
Let $p\in\mathfrak{S}$ and $\ell\neq p$ be a prime. If $d\geq0$ and
$D\geq1$ are integers such that $-d\equiv\Box\pmod{4p}$ and
$D\equiv\Box\pmod{4p}$, then
\begin{eqnarray*}
B(D,\ell^{2n}d)&=&\ell^nB(\ell^{2n}D,d)+\sum_{t=0}^{n-1}
\bigg(\frac{D}{\ell}\bigg)^{n-t-1}(B(D/\ell^2,\ell^{2t}d)-\ell^{t+1}B(\ell^{2t}D,d/\ell^2))\\
&&+\sum_{t=0}^{n-1}\bigg(\frac{D}{\ell}\bigg)^{n-t-1}
\bigg(\bigg(\bigg(\frac{D}{\ell}\bigg)-\bigg(\frac{-d}{\ell}\bigg)\bigg)
\ell^{t}B(\ell^{2t}D,d)\bigg)
\end{eqnarray*}
for all $n\geq1$.
\end{proposition}
\begin{proof}
From the definition (\ref{Heckeoperator}) we have
\begin{eqnarray}
A_\ell(D,d)&=&\ell
A(\ell^2D,d)+\bigg(\frac{D}{\ell}\bigg)A(D,d)+A(D/\ell^2,d)\label{first}\\
B_\ell(D,d)&=&\ell
B(D,d/\ell^2)+\bigg(\frac{-d}{\ell}\bigg)B(D,d)+B(D,d\ell^2).\label{second}
\end{eqnarray}
Combining Proposition \ref{main} with (\ref{first}) we get
\begin{equation}
B_\ell(D,d)=\ell
B(\ell^2D,d)+\bigg(\frac{D}{\ell}\bigg)B(D,d)+B(D/\ell^2,d).\label{third}
\end{equation}
We then derive from (\ref{second}) and (\ref{third}) that
\begin{equation}
B(D,\ell^2d)=\ell
B(\ell^2D,d)+\bigg(\frac{D}{\ell}\bigg)B(D,d)+B(D/\ell^2,d) -\ell
B(D,d/\ell^2)-\bigg(\frac{-d}{\ell}\bigg)B(D,d).\label{last}
\end{equation}
The remaining part of the proof is exactly the same as that of
\cite{Jenkins} Theorem 1.1. Indeed, one can readily prove the
proposition by using induction on $n$ and applying only
(\ref{last}).
\end{proof}

Now, we are ready to prove our main theorem which would be a
generalization of Osburn's result.

\begin{theorem}\label{generalize}
With the same notations as in Theorem \ref{Osburntheorem} we have
\begin{equation*}
t^{(p)}(\ell^{2n}d)\equiv0\pmod{\ell^n}
\end{equation*}
for all $n\geq1$.
\end{theorem}
\begin{proof}
We achieve that
\begin{eqnarray*}
t^{(p)}(\ell^{2n}d)&=&-B(1,\ell^{2n}d)~\textrm{by Proposition
\ref{Jacobi}(iv)}\\
&=&-\ell^nB(\ell^{2n},d)-\sum_{t=0}^{n-1}
\bigg(\frac{1}{\ell}\bigg)^{n-t-1}(B(1/\ell^2,\ell^{2t}d)-\ell^{t+1}B(\ell^{2t},d/\ell^2))\\
&&-\sum_{t=0}^{n-1}\bigg(\frac{1}{\ell}\bigg)^{n-t-1}
\bigg(\bigg(\bigg(\frac{1}{\ell}\bigg)-\bigg(\frac{-d}{\ell}\bigg)\bigg)
\ell^{t}B(\ell^{2t},d)\bigg)~\textrm{by Proposition \ref{Jenkins}}\\
&=&-\ell^nB(\ell^{2n},d)~\textrm{by the facts that $1/\ell^2$ and
$d/\ell^2$ are not integers, and $\bigg(\frac{-d}{\ell}\bigg)=1$}\\
&\equiv&0\pmod{\ell^n}
\end{eqnarray*}
as desired.
\end{proof}

\bibliographystyle{amsplain}

\begin{thebibliography}{10}

\bibitem {EZ} M. Eichler and D. Zagier,
\textit{The Theory of Jacobi Forms}, Progress in Math., vol. 55,
Birkh\"{a}user, Basel (1985).

\bibitem{Jenkins} P. Jenkins, \textit{$p$-adic properties for traces of singular moduli},
Int. J. Number Theory 1 (2005), no. 1, 103-107.

\bibitem{Kim0} C. H. Kim, \textit{Borcherds products associated with certain Thompson
series}, Compos. Math. 140 (2004), no. 3, 541-551.

\bibitem {Kim} C. H. Kim,
\textit{Traces of singular values and Borcherds products}, Bull.
London Math. Soc. 38 (2006), no. 5, 730-740.

\bibitem {Ogg} A. P. Ogg, \textit{Automorphismes de courbes modulaires}, Seminaire
Delange-PisotPoitou (16e annee: 1974/75), Theorie des nombres, Fasc.
1, Exp. no. 7, 8 pp. Secretariat Mathematique, Paris, 1975.

\bibitem {Osburn} R. Osburn, \textit{Congruences for traces of singular
moduli}, Ramanujan J. 14 (2007), no. 3, 411-419.

\bibitem {Zagier} D. Zagier, \textit{Traces of singular moduli}, Motives, polylogarithms
and Hodge theory, Part I (Irvine, CA, 1998), 211-244, Int. Press
Lect. Ser., no. 3, I, Int. Press, Somerville, MA, 2002.

\end{thebibliography}

\end{document}